\documentclass[letterpaper,10pt,pdflatex,reqno]{amsart}
\usepackage{amsmath,times}
\usepackage{amsthm}
\usepackage{amssymb}
\usepackage{latexsym}
\usepackage{amscd}
\usepackage{bbm}
\usepackage[english]{babel}
\usepackage[latin1]{inputenc}
\usepackage{graphicx}
\usepackage{color}
\usepackage{verbatim}

\newcommand{\Q}{\mathbb{Q}}
\newcommand{\R}{\mathbb{R}}

\renewcommand{\a}{\alpha}
\renewcommand{\b}{\beta}

\newcommand{\e}{\varepsilon}

\newcommand{\g}{\gamma}

\newcommand{\ie}{{\rm i.e.\ }} 

\renewcommand{\=}{:=}

\DeclareMathOperator{\MA}{MA}

\DeclareMathOperator{\Amp}{Amp}
\DeclareMathOperator{\vol}{vol}

\numberwithin{equation}{section}       

\newtheorem{prop} {Proposition} [section]
\newtheorem{thm}[prop] {Theorem} 

\newtheorem{lem}[prop] {Lemma}

\newtheorem{prop-def}[prop]{Proposition-Definition}

\newtheorem*{corA}{Corollary A}

\theoremstyle{remark}

\newtheorem*{ackn}{Acknowledgment}

\newtheorem{theorem}{Theorem}[section]
\newtheorem{lemma}[theorem]{Lemma}

\theoremstyle{definition}
\newtheorem{definition}[theorem]{Definition}

\newtheorem{remark}[theorem]{Remark}
\newtheorem{conjecture}[theorem]{Conjecture}

\newtheorem{innercustomthm}{Theorem}
\newenvironment{customthm}[1]
  {\renewcommand\theinnercustomthm{#1}\innercustomthm}
  {\endinnercustomthm}

\begin{document} 

\title{Duality between the pseudoeffective and the movable cone on a projective manifold}
\author{David Witt Nystr\"om \\ \\ { \it with an appendix by S\'ebastien Boucksom}}
\maketitle

\begin{abstract}
We prove a conjecture of Boucksom-Demailly-P\u aun-Peternell, namely that on a projective manifold $X$ the cone of pseudoeffective classes in $H^{1,1}_{\mathbb{R}}(X)$ is dual to the cone of movable classes in $H^{n-1,n-1}_{\mathbb{R}}(X)$ via the Poincar\'e pairing. This is done by establishing a conjectured transcendental Morse inequality for the volume of the difference of two nef classes on a projective manifold. As a corollary the movable cone is seen to be equal to the closure of the cone of balanced metrics. In an appendix by Boucksom it is shown that the Morse inequality also implies that the volume function is differentiable on the big cone, and one also gets a characterization of the prime divisors in the non-K\"ahler locus of a big class via intersection numbers.  
\end{abstract}

\section{Introduction}

In \cite{BDPP} Boucksom-Demailly-P\u aun-Peternell proved that a line bundle on a projective manifold is pseudoeffective iff its degree along any member of a covering family of curves is non-negative. As explained in \cite{BDPP} and \cite{Dem07} this result should be understood in terms of duality of cones. 

Let $X$ be a compact K\"ahler manifold. There are four important cones in $H^{1,1}_{\mathbb{R}}(X)$ that relates to different notions of positivity. The K\"ahler cone $\mathcal{K}$ is the open convex cone given by the set of K\"ahler classes, i.e. the classes that contain a K\"ahler form. The nef cone $\overline{\mathcal{K}}$ is simply the closure of $\mathcal{K}$; a class in $\overline{\mathcal{K}}$ is called nef. The pseudoeffective cone $\mathcal{E}$ is the closed convex cone given by the set of pseudoeffective classes, i.e. the classes that contain a closed positive current. The interior of $\mathcal{E}$ is called the big cone; a class in $\mathcal{E}^{\circ}$ is called big. We have natural inclusions $\mathcal{K}\subset \mathcal{E}^{\circ}$ and $\overline{\mathcal{K}}\subset \mathcal{E}$, and in general these inclusions are strict.

These notions of positivity are compatible with the corresponding positivity notions for line bundles or divisors: a line bundle $L$ is ample/nef/pseudoeffective/big iff the class $c_1(L)$ is K\"ahler/nef/pseudoeffective/big. 

In $H^{n-1,n-1}_{\mathbb{R}}(X)$ there are two important cones. The first cone $\mathcal{N}$, also called the pseudoeffective cone, consists of all classes that contain a closed positive $(n-1,n-1)$-current. The other one is called the the movable cone $\mathcal{M}$ and is defined as the closed convex cone generated by classes of the form $\mu_*(\tilde{\beta}_1\wedge ...\wedge \tilde{\beta}_{n-1})$ where $\mu: \tilde{X}\to X$ is some smooth modification and $\tilde{\beta}_i$ are K\"ahler classes on $\tilde{X}$. The cohomology class associated to a curve in $X$ will lie in $\mathcal{M}$ iff it moves in an analytic family which covers $X$ (see \cite{BDPP}); such a curve is called movable. 

Note that there is a natural pairing (sometimes called the Poincar\'e pairing) between $H^{1,1}(X,\mathbb{R})$ and $H^{n-1,n-1}(X,\mathbb{R})$ given by $(\alpha \cdot \eta):=\int_X \alpha\wedge \eta.$

The following fundamental result was proved by Demailly-P\u aun in \cite{DP}:

\begin{thm}
The nef cone $\overline{\mathcal{K}}$ and pseudoeffective cone $\mathcal{N}$ are dual (with respect to the Poincar\'e pairing). More concretely, a class $\alpha\in H^{1,1}(X,\mathbb{R})$ is nef iff $(\alpha \cdot \eta)\geq 0$ for all $\eta\in \mathcal{N}$.
\end{thm} 

When $X$ is projective we let $\mathcal{E}_{NS}:=\mathcal{E}\cap NS_{\mathbb{R}}(X)$ where $$NS_{\mathbb{R}}(X):=(H^{1,1}_{\mathbb{R}}(X)\cap H^2(X,\mathbb{Z})/\textrm{tors})\otimes_{\mathbb{Z}}\mathbb{R},$$ and similarly $\mathcal{M}_{NS}:=\mathcal{M}\cap N_1(X)$, where $$N_1(X):=(H^{n-1,n-1}_{\mathbb{R}}(X)\cap H^{2n-2}(X,\mathbb{Z})/\textrm{tors})\otimes_{\mathbb{Z}}\mathbb{R}.$$ One can now formulate the result of Boucksom-Demailly-P\u aun-Peternell in \cite{BDPP} in the following way:

\begin{thm} \label{BDPPthm}
On a projective manifold $X$ the cones $\mathcal{E}_{NS}$ and $\mathcal{M}_{NS}$ are dual via the Poincar\'e pairing of $NS_{\mathbb{R}}(X)$ with $N_1(X)$.
\end{thm}   

They also formulated a conjecture:

\begin{conjecture} \label{BDPPconj}
On any compact K\"ahler manifold $X$ the cones $\mathcal{E}$ and $\mathcal{M}$ are dual via the Poincar\'e pairing of $H^{1,1}_{\mathbb{R}}(X)$ with $H^{n-1,n-1}_{\mathbb{R}}(X)$.
\end{conjecture}

More concretely the conjecture says that a class $\alpha\in H_{\mathbb{R}}^{1,1}(X)$ contains a closed positive current iff for all modifications $\mu:\tilde{X}\to X$ of $X$ and K\"ahler classes $\tilde{\beta}_i$ on $\tilde{X}$ we have that 
\begin{equation} \label{numdata}
\int_{\tilde{X}}\mu^*(\alpha)\wedge \tilde{\beta}_1\wedge ...\wedge \tilde{\beta}_{n-1}\geq 0.
\end{equation} 
The if part follows immediately from the fact that one can pull back a closed positive $(1,1)$-current by $\mu$ to get a closed positive $(1,1)$-current on $\tilde{X}$. The (very) hard part is establishing the existence of a closed positive current using the numerical data (\ref{numdata}). 

Our main result confirms the conjecture when $X$ is projective.

\begin{customthm}{A} \label{mainthmA}
When $X$ is projective $\mathcal{E}$ and $\mathcal{M}$ are dual.
\end{customthm}

A Hermitian metric is balanced if given its associated form $\omega$ we have that $\omega^{n-1}$ is closed. Note that $\omega^{n-1}$ is a strongly positive $(n-1,n-1)$-form. Using basic linear algebra one can show that any strongly positive $(n-1,n-1)$-form can be written in a unique way as $\omega^{n-1}$ for some Hermitian metric. The cone of classes of closed strongly positive $(n-1,n-1)$-forms is thus called the balanced cone, denoted by $\mathcal{B}$. It was shown by Toma in \cite{Tom10} for $X$ projective and more generally by Fu-Xiao in \cite{FX14} for $X$  compact K\"ahler that $\mathcal{E}$ and $\overline{\mathcal{B}}$ are dual (for recent work on the balanced cone see e.g. \cite{CRS16} and references therein). As a consequence of our Theorem \ref{mainthmA} we therefore get:

\begin{corA}
When $X$ is projective we have that $\mathcal{M}=\overline{\mathcal{B}}$.
\end{corA} 

Thus at least when $X$ is projective these two a priori different notions of positivity corresponding to the cones $\mathcal{M}$ and $\overline{\mathcal{B}}$ are in fact equivalent.

A key notion in the study of $\mathcal{E}$ is that of volume. 

\begin{definition} \label{defvolume}
The volume of a class $\alpha\in \mathcal{E}$, denoted $\vol(\alpha)$, is defined as the supremum all numbers $(\tilde{\beta}^n)$ where $\mu:\tilde{X}\to X$ is a modification and $\tilde{\beta}$ is a K\"ahler class on $\tilde{X}$ such that $\tilde{\beta}\leq \mu^*(\alpha)$ (i.e. $\mu^*(\alpha)-\tilde{\beta}$ is pseudoeffective). When $\alpha$ is not pseudoeffective we define its volume to be zero.
\end{definition}

It was observed already in \cite{BDPP} that to prove Conjecture \ref{BDPPconj} it is enough to establish a certain lower bound on the volume of the difference of two nef classes $\alpha$ and $\beta$, namely 
\begin{equation} \label{volineqintro}
\vol(\alpha-\beta)\geq (\alpha^n)-n(\alpha^{n-1}\cdot \beta).
\end{equation} 
This inequality is known as a transcendental Morse inequality. The case when $\alpha$ and $\beta$ lies in $NS_{\R}(X)$ is well-known and not hard to prove (see e.g. \cite[Ex. 2.2.33]{Laz04}), and it is used in a crucial way in the proof of Theorem \ref{BDPPthm}. Indeed given the transcendental Morse inequality (\ref{volineqintro}) the rest of the proof of Theorem \ref{BDPPthm} in \cite{BDPP} extends to the general case.

In this paper we prove the transcendental Morse inequality when $X$ is projective.

\begin{customthm}{B} \label{mainthmB}
Let $\alpha$ and $\beta$ be two nef classes on a projective manifold $X$. Then the transcendental Morse inequality holds, i.e. $$\vol(\alpha-\beta)\geq (\alpha^n)-n(\alpha^{n-1}\cdot \beta).$$
\end{customthm}

It was recognized in \cite{BDPP} that to prove Conjecture \ref{BDPPconj} in the case when $X$ is projective, it is enough to establish the Morse inequality (\ref{volineqintro}) for pairs of nef classes $\alpha,\beta$ where $\beta\in NS_{\mathbb{R}}(X)$. Recently Boucksom observed that even the weaker estimate $$\vol(\alpha-\beta)\geq (\alpha^n)-\sum_{k=1}^n \binom{n}{k}(\alpha^{n-k}\cdot \beta^k)$$ would be enough in the argument of \cite{BDPP}, hence yielding the stronger Morse inequality as well as Theorem \ref{mainthmA}.

It is thus this key estimate we prove in this paper.

\begin{prop} \label{propnewest}
Let $X$ be projective, $\alpha,\beta\in H^{1,1}(X,\mathbb{R})$ two nef classes where $\beta\in NS_{\mathbb{R}}(X)$. Then we have that $$\vol(\alpha-\beta)\geq (\alpha^n)-\sum_{k=1}^n \binom{n}{k}(\alpha^{n-k}\cdot \beta^k).$$
\end{prop}

\begin{remark}
In the original version of our paper we established the transcendental Morse inequality directly, not going via Proposition \ref{propnewest}. However, Boucksom's suggestion of using the weaker estimate of Proposition \ref{propnewest} simplifies the proof (while not changing the core of the argument), and hence we adopt that route here.
\end{remark}

The details of how Proposition \ref{propnewest} implies Theorem \ref{mainthmA} and also Theorem \ref{mainthmB} are given in Appendix \ref{Seb}, graciously provided by S\'ebastien Boucksom.

In \cite{BFJ} Boucksom-Favre-Jonsson proved, using the estimate (\ref{volineqintro}), the following two theorems:

\begin{thm} \label{BFJthm1}
The volume function is $C^1$ on $\mathcal{E}_{NS}^{\circ}$ and the partial derivatives are given by $$\frac{d}{dt}\bigg|_{t=0}\vol(\alpha+t\gamma)=n\langle\alpha^{n-1}\rangle \cdot \gamma.$$
\end{thm}

\begin{thm} \label{BFJthm2}
For any class big class $\alpha\in NS_{\R}$ we have that $$\vol(\alpha)=\langle \alpha^{n-1}\rangle \cdot \alpha.$$ In particular a prime divisor $D$ will lie in the augmented base locus of $L$ iff $$\langle c_1(L)^{n-1} \rangle \cdot c_1(D)=0.$$ 
\end{thm}

Here $\langle \alpha^{n-1}\rangle$ denotes a positive selfintersection of $\alpha$ (see Section \ref{secMA}), which is equal to $\alpha^{n-1}$ when $\alpha$ is nef but not in general. This last result can be thought of as an orthogonality relation (see \cite{BFJ} and Appendix \ref{Seb}).

Appendix \ref{Seb} by Boucksom shows that the analogues of Theorem \ref{BFJthm1} and Theorem \ref{BFJthm2} also are consequences of Proposition \ref{propnewest}, hence we get:   

\begin{customthm}{C} \label{mainthmC}
On a projective manifold $X$ the volume function is continuously differentible on the big cone with  $$\frac{d}{dt}\bigg|_{t=0}\vol(\alpha+t\gamma)=n\langle\alpha^{n-1}\rangle \cdot \gamma.$$ 
\end{customthm}

\begin{customthm}{D} \label{mainthmD}
For any big class $\alpha\in \mathcal{E}^{\circ}$ on a projective manifold $X$ we have that $$\vol(\alpha)=\langle \alpha^{n-1}\rangle \cdot \alpha.$$ In particular a prime divisor $D$ will lie in the non-K\"ahler locus of $\alpha$ iff $$\langle \alpha^{n-1} \rangle \cdot c_1(D)=0.$$
\end{customthm}

\subsection{Related work}

As has already been said Boucksom-Demailly-P\u aun-Peternell proved the integral version of Conjecture \ref{BDPPconj} in \cite{BDPP}. They also settled the conjecture in the case when $X$ is compact hyperk\"ahler, or more generally, a limit by deformation of projective manifolds with Picard number $\rho=h^{1,1}$ (see \cite{BDPP}, Corollary 10). In the projective case they established a weaker version of Theorem \ref{mainthmB}, namely that if $\alpha$ is nef and $\beta$ is nef and lies in $NS_{\R}(X)$ then $$\vol(\alpha-\beta)\geq (\alpha^n)-\frac{(n+1)^2}{4}(\alpha^{n-1}\cdot \beta).$$ Their proof is different from ours. As ours it uses a family of $\theta$-psh functions that converge to something with a logarithmic singularity along the divisor of $\omega$. But instead of being envelopes these functions solve a Monge-Amp\`ere equation, which concentrates the mass along the divisor.

In \cite{X1} Xiao proved a kind of weaker qualitative version of (\ref{volineqintro}), namely given two nef classes $\alpha$ and $\beta$ on a compact K\"ahler manifold $X$ then if $$(\alpha^n)>4n(\alpha^{n-1}\cdot\beta)$$ it follows that $\alpha-\beta$ is big. Later, using the same kind of techniques as Xiao, Popovici improved on this, showing that $$(\alpha^n)>n(\alpha^{n-1}\cdot\beta)$$ implies $\alpha-\beta$ to be big. Then Xiao refining the techniques further established in \cite{X2} that if $\alpha$ is big and $\beta$ movable then $$\vol(\alpha)>n(\langle \alpha^{n-1}\rangle \cdot\beta)$$ implies $\alpha-\beta$ to be big. While these results are qualitative the recent work \cite{P2} by Popovici gives quantitative results as well, namely that if $\alpha$ and $\beta$ are nef and $$\vol(\alpha-\beta)\geq ((\alpha-\beta)^n)$$ then $$\vol(\alpha-\beta)\geq (\alpha^n)-n(\alpha^{n-1}\cdot \beta).$$  

The differentiability of the volume of big line bundles was proved independently and at the same time as Boucksom-Favre-Jonsson by Lazarsfeld-Musta\c{t}\u{a} \cite{LM} using the theory of Okounkov bodies. They expressed the derivative in terms of restricted volumes, and as a consequence the restricted volume of a big line bundle $L$ along a prime divisors $D$ coincides with the pairing $\langle c_1(L)^{n-1} \rangle\cdot c_1(D)$. The restricted volume is only really defined along subvarieties that are not contained in the augmented base locus, while the pairing $\langle c_1(L)^{n-1} \rangle\cdot c_1(D)$ always is defined and furthermore depend continuously on $L$. It thus follows from Theorem \ref{BFJthm2} that if a prime divisor $D$ is not contained in the augmented base locus  for small ample perturbations $L+\epsilon A$ of $L$ and furthermore the restricted volume of $L+\epsilon A$ along $D$ remains bounded from below by some positive number then $D$ cannot be contained in the augmented base locus of $L$. A deep result of Ein-Lazarsfeld-Musta\c{t}\u{a}-Nakamaye-Popa \cite{ELMNP} states that the restricted volume can be used to characterize the whole augmented base locus. Whether the analogous result pertaining to the non-K\"ahler locus of a big class is true is still non known, but we note that our Theorem \ref{mainthmD} can be seen as a partial result in that direction. The nef case though was recently completely settled by Collins-Tosatti \cite{CT}. They prove that if $\alpha$ is a nef and big class on a compact K\"ahler manifold then the non-K\"ahler locus of $\alpha$ is equal to the null-locus, i.e. union of all irreducible analytic subspaces $V$ such that $$\int_V\alpha^{\dim V}=0.$$

\begin{ackn}
I would like to thank Robert Berman, Bo Berndtsson, Jean-Pierre Demailly, Mihai P\u aun, Valentino Tosatti and Jian Xiao for fruitful discussions on this topic and valuable comments on an early draft of this paper. I thank Chinh Lu for suggesting a simplification of the proof. I am particularly grateful for the suggestions of simplifications given by S\'ebastien Boucksom, and especially for him writing the appendix, thus providing the deriviation of Theorem \ref{mainthmA}, \ref{mainthmB}, \ref{mainthmC} and \ref{mainthmD} from Proposition \ref{propnewest}.  
\end{ackn}

\section{Preliminaries} \label{Prel}

\subsection{$\theta$-psh functions}

Let $X$ be a compact K\"ahler manifold, $\theta$ a closed smooth real $(1,1)$-form on $X$ and $\alpha:=[\theta]\in H^{1,1}_{\mathbb{R}}(X)$ its cohomology class. We say that a function $u: X\to [-\infty,\infty)$ is $\theta$-psh if whenever locally $\theta=dd^cv$ for some smooth function $v$ we have that $u+v$ is plurisubharmonic and not identically equal to $-\infty$. Thus $\theta+dd^cu$ is a closed positive $(1,1)$-current. Conversely, if $T$ is a closed positive $(1,1)$-current in $\alpha$ then there exists a $\theta$-psh function $u$ such that $T=\theta+dd^cu$, and this $u$ is unique up to a constant. 

We say that $u$ is strictly $\theta$-psh if it is $\theta-\epsilon\omega$-psh for some $\epsilon>0$ and K\"ahler form $\omega$.

If $\theta'$ is another closed smooth real $(1,1)$-form cohomologuous to $\theta$, then by the $dd^c$-lemma there exists a smooth function $f$ such that $\theta'=\theta+dd^c f$. Thus one sees that $u$ is $\theta$-psh iff $u-f$ is $\theta'$-psh.

The set of $\theta$-psh functions is denoted by $PSH(X,\theta)$. The class $\alpha$ is called pseudoeffective if it contains a closed positive current and we  note that this is equivalent to $PSH(X,\theta)$ being nonempty. A class is said to be big if for some $\epsilon>0$ and some K\"ahler class $\beta$ we have that $\alpha-\epsilon\beta$ is pseudoeffective.

A $\theta$-psh function $u$ is said to have analytic singularities if locally it can be written as $c\ln(\sum_i |g_i|^2)+f$ where $c>0,$ $g_i$ is a finite collection of local holomorphic functions and $f$ is smooth. A deep regularization result of Demailly states that if $\alpha$ is big then there are strictly $\theta$-psh functions with analytic singularities. 

\begin{definition}
If $\alpha$ is big we say that a point $x\in X$ lies in the ample locus of $\alpha$, denoted by $\Amp(\a)$, if there exists a striclty $\theta$-psh function with analytic singularities which is smooth near $x$. The complement of $\Amp(\alpha)$ is called the non-K\"ahler locus of $\alpha$, denoted $E_{nK}(\alpha)$, and we note that $E_{nK}(\alpha)$ will be a proper analytic subset of $X$.
\end{definition}

\begin{remark}
When $X$ is projective and $\alpha=c_1(L)$ for some holomorphic line bundle $L$ then $E_{nK}(\alpha)$ coincides with the augmented base locus of $L$ (see e.g. \cite{BEGZ}).
\end{remark}

A $\theta$-psh function $u$ is said to have minimal singularities if for every $v\in PSH(X,\theta)$ we have that $u\geq v+O(1)$. It is easy to show using envelopes that whenever $\alpha$ is pseudoeffective one can find $\theta$-psh functions with minimal singularities. They are far from unique though, in fact this is what we will exploit later in the proof of Theorem \ref{mainthmB}.

It is clear that if $\alpha$ is big and $u\in PSH(X,\theta)$ has minimal singularities then $u$ is locally bounded on $\Amp(\alpha)$.

\subsection{Monge-Amp\`ere measures and positive intersections} \label{secMA}

A key tool will be the notion of the Monge-Amp\`ere measure of a psh or $\theta$-psh function. This theory was first developed in the local setting by Bedford-Taylor  \cite{BT} and later in the geometric setting of compact K\"ahler manifolds by Boucksom-Eyssidieux-Guedj-Zeriahi in \cite{BEGZ}.

We will start discussing the local picture. Let $u$ be a psh function on $U$ where $U$ is some domain in $\mathbb{C}^n$. If $u$ is smooth then $\MA(u):=(dd^cu)^n$ is a positive measure called the Monge-Amp\`ere of $u$. However, in general $dd^cu$ is a form with measure coefficients, and since the multiplication of measures typically is illdefined $(dd^cu)^n$ might not make sense. Bedford-Taylor showed in \cite{BT} that if one assumes $u$ to be locally bounded then one can define $(dd^cu)^n$ inductively by $(dd^cu)^{k+1}:=dd^c(u(dd^cu)^k)$ and they proved that $\MA(u):=(dd^cu)^n$ will be a positive measure, still called the Monge-Amp\`ere of $u$. 

Bedford-Taylor also proved some fundamental continuity properties of the Monge-Amp\`ere operator (see \cite{BT}). Here we will only mention one:

\begin{thm} \label{BTthm}
Let $u_k$ be a decreasing sequence of psh functions on $U$ such that $u:=\lim_{k\to \infty }u_k$ is locally bounded on $U$ ($u$ will then be psh). Then the Monge-Amp\`ere measures $\MA(u_k)$ converge weakly to $\MA(u)$.  
\end{thm}

Another important fact proven by Bedford-Taylor is that the Monge-Amp\`ere measure $\MA(u)$ of a locally bounded psh function $u$ never puts mass on proper analytic subsets (or more generally proper pluripolar subsets).

Now we come to the global picture, explored in the work of Guedj-Zeriahi (see e.g. \cite{GZ05,GZ07}) and later Boucksom-Eyssidieux-Guedj-Zeriahi (see \cite{BEGZ}).

Let $\theta$ be a closed smooth real $(1,1)$-form on a compact K\"ahler manifold $X$ and and assume that $\alpha:=[\theta]\in H^{1,1}_{\mathbb{R}}(X)$ is big. If $u$ is $\theta$-psh with minimal singularities it is locally bounded on $\Amp(\alpha)$ and we define the (nonpluripolar) Monge-Amp\`ere measure of $u$ (with respect to $\theta$) as $$\MA_{\theta}(u):=\mathbbm{1}_{\Amp(\alpha)}(dd^cu+\theta)^n.$$ Note that as in the local case $\MA_{\theta}(u)$ never puts mass on proper analytic subsets.

More generally, Boucksom-Eyssidieux-Guedj-Zeriahi shows that for any $p\in \{1,...,n\}$ one can define a positive current $$\langle (dd^cu+\theta)^p\rangle:=\mathbbm{1}_{\Amp(\alpha)}(dd^cu+\theta)^p,$$ which also will be closed by the Skoda-El Mir theorem (see \cite{BEGZ}).

From the fundemantal \cite[Thm. 1.16]{BEGZ} follows the next crucial result:

\begin{thm} \label{BEGZthm}
If $\alpha$ is big and $u\in PSH(X,\theta)$ has minimal singularities then the cohomology class of $\langle (dd^cu+\theta)^p\rangle$ is independent of $u$; this class is thus denoted by $\langle \alpha^p\rangle$. For $p=n$ we have that $\langle \alpha^n\rangle=\vol(\alpha)$, or in other words 
\begin{equation} \label{eqforvolume}
\vol(\alpha)=\int_X \MA_{\theta}(u).
\end{equation}
\end{thm}

\begin{remark}
In \cite{BEGZ} they choose to define $\vol(\alpha)$ using (\ref{eqforvolume}), but \cite[Thm. 1.16]{BEGZ} shows that this definition coincides with our Definition \ref{defvolume}. Similarly,  in \cite{BDPP} and \cite{BFJ} the positive intersections $\langle \alpha^p\rangle$ are defined using modifications, but \cite[Thm. 1.16]{BEGZ} shows the equivalence of these different definitions.
\end{remark}

We will also need the following deep result of Boucksom in \cite{B1}, building on work of Demailly-P\u aun in \cite{DP} (see also \cite{C}).

\begin{thm} \label{Bthm}
Let $\alpha_k$ be a sequence of big classes that converge to a class $\alpha$. Then if $\limsup_{k\to \infty}\vol(\alpha_k)>0$ it follows that $\alpha$ is big. 
\end{thm}    

From this we see that letting $\vol(\alpha):=0$ for $\alpha$ not big gives a continuous extension of the volume function to the whole of $H_{\mathbb{R}}^{1,1}(X)$. 

\section{Regularity of envelopes}

In our proof of Theorem \ref{mainthmB} a key role will be played by a family of envelopes. The proof will rely on us being able to control the behaviour of their Monge-Amp\`ere measures. For this we need a deep result of Berman-Demailly \cite{BD}.

\begin{thm} \label{BDThm}
Let $\theta$ be a smooth closed real $(1,1)$-form on a compact K\"ahler manifold $(X,\omega)$. Assume that the class $\alpha:=[\theta]$ is big and let $\psi_0$ be a strictly $\theta$-psh function with analytic singularities. Let $\phi$ be defined as $$\phi:=\sup\{\psi\leq 0: \psi\in PSH(X,\theta)\},$$ and let $D:=\{\phi=0\}$. Then $\phi\in PSH(X,\theta)$ has minimal singularities and for some constants $C$ and $B$ we have that $$|dd^c\phi|_{\omega}\leq C(|\psi_0|+1)^2e^{B|\psi_0|}.$$ It follows that $$\MA_{\theta}(\phi)=\mathbbm{1}_D\theta^n$$ and hence $$\vol(\alpha)=\int_X \MA_{\theta}(\phi)=\int_D\theta^n.$$  
\end{thm}

\begin{remark}
We will actually only need the case when $\alpha$ is K\"ahler (then $\psi_0$ can be chosen to be smooth). This simpler case was given an alternative proof by Berman in \cite{Ber13}.
\end{remark}

If $f$ is a smooth function we can also consider the envelope $$\phi_f:=\sup\{\psi\leq f: \psi\in PSH(X,\theta)\}.$$ It is easy to see that $$\phi_f-f=\sup\{\psi\leq 0: \psi\in PSH(X,\theta+dd^cf)\},$$ and thus by Theorem \ref{BDThm} $$\vol(\alpha)=\int_X \MA_{\theta}(\phi_f)=\int_{D_f}\theta^n,$$ where $D_f:=\{\phi_f=f\}$.

\section{Proof of Proposition \ref{propnewest}}

\begin{proof}
Recall that we need to show that if $\alpha,\beta\in H^{1,1}(X,\mathbb{R})$ are two nef classes on a projective manifold $X$, and also $\beta\in NS_{\mathbb{R}}(X)$, then we have that 
\begin{equation} \label{weakeqproof}
\vol(\alpha-\beta)\geq (\alpha^n)-\sum_{k=1}^n \binom{n}{k}(\alpha^{n-k}\cdot \beta^k).
\end{equation}

We can assume that $$(\alpha^n)-\sum_{k=1}^n \binom{n}{k}(\alpha^{n-k}\cdot \beta^k)>0,$$ because otherwise the inequality (\ref{weakeqproof}) would be trivially fulfilled.

Assume that we can show the inequality (\ref{weakeqproof}) under the additional assumption that $\alpha-\beta$ is big. Then for general $\alpha$ and $beta$ we can let $t_0:=\sup\{t\in [0,1]: \alpha-t\beta \textrm{ is big}\}$. By the continuity and monotonicity of the volume we get that 
\begin{eqnarray*}
\vol(\alpha-t_0\beta)=\lim_{t\to t_0}\vol(\alpha-t_0\beta)\geq  (\alpha^n)-\sum_{k=1}^n \binom{n}{k}t_0^k(\alpha^{n-k}\cdot \beta^k)\geq \\ \geq  (\alpha^n)-\sum_{k=1}^n \binom{n}{k}(\alpha^{n-k}\cdot \beta^k)>0.
\end{eqnarray*} 
This shows that $\alpha-t_0\beta$ is big and thus $t_0=1$, which therefore implies the desired inequality (\ref{weakeqproof}). 

Thus without loss of generality we can assume that $\alpha-\beta$ is big. By the homogeneity and continuity of the volume and the intersection numbers we can also without loss of generality assume that $\alpha$ is K\"ahler while $\beta=c_1(L)$ for some very ample holomorphic line bundle $L$ (which also means that $\beta$ is K\"ahler).

Let $\theta$ and $\omega$ be K\"ahler forms in $\alpha$ and $\beta$ respectively, and let $s$ be a nontrivial holomorphic section of $L$. There is a positive metric $h$ of $L$ whose curvature form is $\omega$, and we let $$g:=\ln|s|^2_h,$$ where we normalize $h$ so that $\max g=0$. We thus have $dd^cg=[Y]-\omega$ where $Y$ is the effective divisor defined by $s$. 

The idea of the proof is to go between $PSH(X,\theta)$ and $PSH(X,\theta-\omega)$ using the function $g$. How this works is described in the following lemma.

\begin{lemma} \label{lemtrans}
If $v\in PSH(X,\theta-\omega)$ then $v+g\in PSH(X,\theta)$. Conversely, if $u\in PSH(X,\theta)$ and $u\leq g+O(1)$ then $u-g$ extends across $Y$ to a $\theta-\omega$-psh function on $X$.
\end{lemma}

\begin{proof}
If $v\in PSH(X,\theta-\omega)$ then $$dd^c(v+g)+\theta=dd^cv+[Y]-\omega+\theta\geq dd^cv +(\theta-\omega)\geq 0$$ by assumption, and so $v+g\in PSH(X,\theta)$. On the other hand, if $u\in PSH(X,\theta)$ then on $X\setminus Y$, since there $dd^cg=-\omega$, we have that $$dd^c(u-g)+(\theta-\omega)=dd^cu+\theta\geq 0$$ by assumption. If also $u\leq g+O(1)$ then $u-g$ is bounded so it extends across the analytic set $Y$ as a $\theta-\omega$-psh function.
\end{proof}

Let $|\cdot|_{reg}$ be a smooth convex function on $\mathbb{R}$ which coincides with $|\cdot|$ for $|x|\geq 1$. We then let $$\max_{reg}(x,y):=\frac{x+y+|x-y|_{reg}}{2}$$ be the corresponding regularized max function. 

For any $R>0$ we define $$g_R:=\max_{reg}(g,-R).$$ Since $g\in PSH(X,\omega)$, $-R\in PSH(X,\omega)$ and $\max_{reg}$ is a convex function it follows that $g_R\in PSH(X,\omega)$. We also note that $g_R$ decreases to $g$ as $R\to \infty$.

Now let $$\phi_R:=\sup\{\psi\leq g_R: \psi\in PSH(\theta)\}$$ and $$D_R:=\{\phi_R=g_R\}.$$ Since $\phi_R$ is bounded (i.e. it has minimal singularities) it follows from Theorem \ref{BEGZthm} that 
\begin{equation} \label{voleq1}
(\alpha^n)=\vol(\alpha)=\int_X\MA_{\theta}(\phi_R).
\end{equation} 
We also get from Theorem \ref{BDThm} and the accompanying remark that 
\begin{equation} \label{MAmain}
\MA_{\theta}(\phi_R)=\mathbbm{1}_{D_R}(\theta+dd^cg_R)^n.
\end{equation}  

Since $g_R$ is decseasing we get that $\phi_R$ is decreasing and so $\phi_{\infty}:=\lim_{R\to\infty}\phi_R$ is $\theta$-psh unless it is identically equal to $-\infty$. 

\begin{lemma} \label{lemminsing}
$\phi_{\infty}$ is not  identically equal to $-\infty$ and $\phi_{\infty}-g$ is $(\theta-\omega)$-psh with minimal singularities. 
\end{lemma}

\begin{proof}
Let $v\in PSH(X,\theta-\omega)$ have minimal singularities, and by adding some constant we can assume that $v\leq 0$. From Lemma \ref{lemtrans} we see that $v+g\in PSH(X,\theta)$. Note that $v+g\leq g\leq g_R$ for all $R$. It follows that $v+g\leq \phi_R$ for all $R$ and hence $v+g\leq \phi_{\infty}$. It follows that $\phi_{\infty}$ is not  identically equal to $-\infty$.  By the second part of Lemma \ref{lemtrans} we get that $\phi_{\infty}-g$ lies in $PSH(X,\theta-\omega)$ and since $\phi_{\infty}-g\geq v$ where $v$ had minimal singularities we see that $\phi_{\infty}-g$ also has minimal singularities.
\end{proof}

It now follows from Theorem \ref{BEGZthm} that 
\begin{equation} \label{voleq}
\vol(\alpha-\beta)=\int_X \MA_{\theta-\omega}(\phi_{\infty}-g).
\end{equation} 

Note that away from $Y$ we have that $dd^cg=-\omega$. It follows that 
\begin{eqnarray*}
\mathbbm{1}_{X\setminus Y}\MA_{\theta-\omega}(\phi_{\infty}-tg)=\mathbbm{1}_{\Amp(\a-\b)\setminus Y}(dd^c(\phi_{\infty}-g)+\theta-\omega)^n=\\=\mathbbm{1}_{\Amp(\a-\b)\setminus Y}(dd^c\phi_{\infty}+\theta)^n=\mathbbm{1}_{\Amp(\a-\b)\setminus Y}\MA_{\theta}(\phi_{\infty}).
\end{eqnarray*} 
Since Monge-Amp\`ere measures put no mass on proper analytic subsets such as $E_{nK}(\alpha)\cup Y$ this shows that $$\MA_{\theta-\omega}(\phi_{\infty}-g)=\MA_{\theta}(\phi_{\infty}).$$ Combined with (\ref{voleq}) it means that $$\vol(\alpha-\beta)=\int_X \MA_{\theta}(\phi_{\infty}).$$

By Lemma \ref{lemminsing} $\phi_{\infty}$ is locally bounded $\Amp(\a-\b)\setminus Y$. By Theorem \ref{BTthm} the measures $\MA_{\theta}(\phi_R)$ converge weakly to $\MA_{\theta}(\phi_R)$ on $\Amp(\a-\b)\setminus Y$.  Picking a large open set $U$ such that $\overline{U}\subseteq (\Amp(\a-\b)\setminus Y)$ we therefore get
\begin{eqnarray*}
\vol(\alpha-\beta)=\int_X \MA_{\theta}(\phi_{\infty})\geq \int_{\overline{U}}\MA_{\theta}(\phi_{\infty})\geq \lim_{R\to\infty}\int_U \MA_{\theta}(\phi_R).
\end{eqnarray*} 
Combined with (\ref{voleq1}) and (\ref{MAmain}) this yields 
\begin{equation} \label{finaleq1}
\vol(\alpha-\beta)\geq (\alpha^n)-\lim_{R\to\infty}\int_{U^c} \MA_{\theta}(\phi_R)=(\alpha^n)-\lim_{R\to\infty}\int_{D_R\cap U^c}(\theta+dd^cg_R)^n.
\end{equation}

Recall that for all $R,$ $g_R\in PSH(X,\omega)$. We thus get the estimate $$\mathbbm{1}_{D_R}(\theta+dd^cg_R)^n\leq (\theta+(dd^cg_R+\omega))^n=\sum_{k=0}^n \binom{n}{k}\theta^{n-k}\wedge (dd^cg_R+\omega)^k,$$ where each term on the right hand side is a positive measure. It follows that 
\begin{eqnarray} \label{finaleq2}
\int_{D_R\cap U^c}(\theta+dd^cg_R)^n\leq \int_{U^c}\theta^n+\sum_{k=1}^n\binom{n}{k}\int_X\theta^{n-k}\wedge (dd^cg_R+\omega)^k= \nonumber \\=\int_{U^c}\theta^n+\sum_{k=1}^n \binom{n}{k}(\alpha^{n-k}\cdot \beta^k).
\end{eqnarray}
Combining (\ref{finaleq1}) and (\ref{finaleq2}) thus yields $$\vol(\alpha-\beta)\geq (\alpha^n)-\int_{U^c}\theta^n-\sum_{k=1}^n \binom{n}{k}(\alpha^{n-k}\cdot \beta^k).$$ Since $\theta^n$ is a volume form it puts no mass on $E_{nK}(\alpha)\cup Y$, so by making $U$ large the term $\int_{U^c}\theta^n$ gets arbitrarily small. This then gives us the desired estimate (\ref{weakeqproof}). 

\end{proof}

\appendix

\section{Remarks on orthogonality, differentiabiltity and duality -- S. Boucksom} \label{Seb}

\subsection{Differentiability and duality in the projective case}
Demailly conjectures that the following 'transcendental Morse inequality' 
\begin{equation}\label{equ:Morse}
\vol(\a-\b)\ge(\a^n)-n(\a^{n-1}\cdot\b)
\end{equation}
holds for any two nef classes $\a,\b\in H_{\mathbb{R}}^{1,1}(X)$ on a compact K\"ahler manifold $X$ of complex dimension $n$. 

In the main paper the following was proved:

\begin{prop}\label{thm:david}
Let $X$ be projective, $\alpha,\beta\in H^{1,1}(X,\mathbb{R})$ two nef classes where $\beta\in NS_{\mathbb{R}}(X)$. Then we have that $$\vol(\alpha-\beta)\geq (\alpha^n)-\sum_{k=1}^n \binom{n}{k}(\alpha^{n-k}\cdot \beta^k).$$
\end{prop}

As we shall see, this result implies the following general statements. 

\begin{thm}\label{thm:proj} Let $X$ be a projective manifold. 
\begin{itemize}
\item[(i)] The Morse inequality (\ref{equ:Morse}) holds for arbitrary nef $(1,1)$-classes. 
\item[(ii)] The differentiability theorem of \cite{BFJ} holds for all $(1,1)$-classes: for each $\a,\g\in H_{\R}^{1,1}(X)$ with $\a$ big, we have 
$$
\frac{d}{dt}\bigg|_{t=0}\vol(\a+t\g)=n\,\g\cdot\langle\a^{n-1}\rangle.
$$
\item[(iii)] The cones $\mathcal{E}$ and $\mathcal{M}$ are dual. 
\end{itemize}
\end{thm}

\subsection{From orthogonality to differentiability}

As we next show, the orthogonality property of \cite{BDPP} is equivalent to the differentiability property of \cite{BFJ}. Our argument is inspired by the simplified proof of \cite[Theorem B]{BB} provided in \cite[Lemma 6.13]{LN} (see also \cite[Proposition 1.1]{X2} for a related result). 

\begin{thm}\label{thm:equiv} For a given compact K\"ahler manifold $X$, the following properties are equivalent: 
\begin{itemize}
\item[(i)] Orthogonality: each big class $\a\in H_{\R}^{1,1}(X)$ satisfies
$$
\vol(\a)=\a\cdot\langle\a^{n-1}\rangle.
$$ 
\item[(ii)] Differentiability: for each $\a,\g\in H_{\R}^{1,1}(X)$ with $\a$ big, we have 
\begin{equation} \label{equ:diff1}
\frac{d}{dt}\bigg|_{t=0}\vol(\a+t\g)=n\,\g\cdot\langle\a^{n-1}\rangle.
\end{equation}
\end{itemize}
Further, these properties imply the transcendental Morse inequality (\ref{equ:Morse}) for all nef classes, as well as the duality between $\mathcal{E}$ and $\mathcal{M}$. 
\end{thm}

\begin{lem}\label{lem:diff} The differentiability property (\ref{equ:diff1}) holds if and only if
\begin{equation}\label{equ:concave}
\vol(\a)^{1/n}-\vol(\b)^{1/n}\ge\frac{(\a-\b)\cdot\langle\a^{n-1}\rangle}{\vol(\a)^{1-1/n}}
\end{equation}
for any two big classes $\a,\b\in H_{\R}^{1,1}(X)$.
\end{lem}
\begin{proof} Since $\vol$ is positive on the big cone, (\ref{equ:diff1}) is equivalent to 
\begin{equation}\label{equ:diff}
\frac{d}{dt}\bigg|_{t=0}\vol(\a+t\g)^{1/n}=\frac{\g\cdot\langle\a^{n-1}\rangle}{\vol(\a)^{1-1/n}}.
\end{equation}
By concavity of $\vol^{1/n}$ on the big cone \cite{B1}, we thus see that (\ref{equ:diff1}) implies (\ref{equ:concave}). Assume conversely that the latter holds. Then 
$$
\frac{t\g\cdot\langle\a^{n-1}\rangle}{\vol(\a)^{1-1/n}}\ge\vol(\a+t\g)^{1/n}-\vol(\a)^{1/n}\ge\frac{t\g\cdot\langle(\a+t\g)^{n-1}\rangle}{\vol(\a+t\g)^{1-1/n}},
$$
for $|t|\ll 1$, which yields (\ref{equ:diff}) by continuity of positive intersection products on the big cone \cite{BFJ}. 
\end{proof}
Since $\vol(\a)=(\a^n)$ is differentiable when $\a$ is nef, the same argument shows that (\ref{equ:concave}) holds when $\a,\b$ are nef and big.

\begin{proof}[Proof of Theorem~\ref{thm:equiv}] Assume that (i) holds, and pick big classes $\a,\b\in H_{\R}^{1,1}(X)$. By Lemma~\ref{lem:diff}, it will be enough to establish (\ref{equ:concave}). By definition of positive intersection numbers, there exists a sequence of modifications $\mu_k:X_k\to X$ and K\"ahler classes $\a_k,\b_k$ on $X_k$ with
\begin{itemize}
\item $\a_k\le\mu_k^*\a$ (\ie the difference is psef); 
\item $\b_k\le\mu_k^*\b$;
\item $\vol(\a_k)\to\vol(\a)$; 
\item $\vol(\b_k)\to\vol(\b)$; 
\item $(\mu_k^*\b\cdot\a_k^{n-1})\to\b\cdot\langle\a^{n-1}\rangle$. 
\end{itemize}
As noted above, (\ref{equ:concave}) holds when the classes are nef, and hence
$$
\vol(\a_k)^{1/n}-\vol(\b_k)^{1/n}\ge\frac{(\a_k-\b_k)\cdot\a_k^{n-1}}{\vol(\a_k)^{1-1/n}}
$$
$$
\ge\frac{(\a_k^n)-(\mu_k^*\b\cdot\a_k^{n-1})}{\vol(\a_k)^{1-1/n}}
$$
since $\b_k\le\mu_k^*\b$ and $\a_k$ is nef. In the limit as $k\to\infty$ we infer
$$
\vol(\a)^{1/n}-\vol(\b)^{1/n}\ge\frac{\vol(\a)-\b\cdot\langle\a^{n-1}\rangle}{\vol(\a)^{1-1/n}},
$$
which is (\ref{equ:concave}) since $\vol(\a)=\a\cdot\langle\a^{n-1}\rangle$. This proves (i)$\Longrightarrow$(ii). Conversely, applying (ii) with $\g=\a$ yields (i), as already observed in \cite{BFJ}. 

Assume now that (i) and (ii) hold. As observed in \cite{BFJ}, the Morse inequality (\ref{equ:Morse}) holds for any two nef classes $\a,\b$, because
$$
\vol(\a-\b)-(\a^n)=-n\int_0^1\b\cdot\langle(\a-t\b)^{n-1}\rangle dt\ge-n(\b\cdot\a^{n-1})
$$
since $\a-t\b\le\a$ and $\b$ is nef. 

We next turn to the duality theorem. It is enough to show that any psef class $\a$ in the interior of the dual of $\mathcal{M}$ is big. For such a class $\a$, there exists a K\"ahler class $\sigma$ on $X$ such that 
\begin{equation}\label{equ:dual}
\a\cdot\langle\b^{n-1}\rangle\ge\sigma\cdot\langle\b^{n-1}\rangle
\end{equation}
for all big classes $\b\in H_{\R}^{1,1}(X)$. For each $\e>0$, $\a+\e\sigma$ is big, and (i) and (\ref{equ:dual}) give 
$$
\vol(\a+\e\sigma)=(\a+\e\sigma)\cdot\langle(\a+\e\sigma)^{n-1}\rangle\ge\a\cdot\langle(\a+\e\sigma)^{n-1}\rangle\ge\sigma\cdot\langle(\a+\e\sigma)^{n-1}\rangle,
$$
and hence 
$$
\vol(\a+\e\sigma)\ge(\sigma^n)^{1/n}\vol(\a+\e\sigma)^{1-1/n}
$$
by the Khovanskii-Teissier inequality (see e.g. \cite[Thm. 3.5(iii)]{BDPP}). This yields $\vol(\a+\e\sigma)\ge(\sigma^n)>0$ for any $\e>0$, which proves that $\a$ is big by \cite{B1}. 
\end{proof}

\begin{proof}[Proof of Theorem~\ref{thm:proj}] By Theorem~\ref{thm:equiv}, it is enough to show that any big class $\a\in H_{\R}^{1,1}(X)$ satisfies $\vol(\a)=\a\cdot\langle\a^{n-1}\rangle$. We do this by adapting the arguments of \cite[\S 4]{BDPP}. As above, we may choose a sequence of approximate Zariski decompositions $\mu_j^*\a=\a_k+E_j$ where $\mu_j:X_k\to X$ is a projective modification, $\a_j$ is K\"ahler, $E_j$ is (the class of) an effective $\Q$-divisor, in such a way that $\langle\a^n\rangle=\lim_{j\to\infty}(\a_j^n)$ and $\langle\a^{n-1}\rangle=\lim_{j\to\infty}(\mu_j)_*\left(\a_j^{n-1}\right)$. Property (i) is then equivalent to the asymptotic orthogonality property $\lim_{j\to \infty}\left(\a_j^{n-1}\cdot E_j\right)=0$ (hence the chosen terminology!). 

Let $H$ be an ample divisor class on $X$ such that $H-\a\in H_{\R}^{1,1}(X)$ is nef. As observed in \cite[\S 10]{BDPP}, the class 
$$
\mu_j^*H-E_j=\mu_j^*(H-\a)+\a_j
$$ 
is nef and rational. For each $t\in[0,1]$, we have
$$
\a_j+t E_j=(\a_j+t\mu_j^*H)-t(\mu_j^*H-E_j)=(\a_j+t\mu_j^*H)-t(\mu_j^*(H-\a)+\a_j)
$$
with $\a_j+t\mu_j^*H\in H_{\R}^{1,1}(X_k)$ nef, and Proposition~\ref{thm:david} therefore yields
\begin{eqnarray} \label{volestapp}
\vol(\a)\ge\vol(\a_j+t E_j)\ge \left((\a_j+t\mu_j^*H)^n\right)- \nonumber\\-nt\left((\a_j+t\mu_j^*H)^{n-1}\cdot (\mu_j^*(H-\a)+\a_j)\right)- \nonumber\\-\sum_{k=2}^n \binom{n}{k}t^k\left((\a_j+t\mu_j^*H)^{n-k}\cdot (\mu_j^*(H-\a)+\a_j)^k\right).
\end{eqnarray}

By assumption $\a_j\leq \mu_j^*H$  and thus for any $k,l,m\in \mathbb{N}$ with $k+l+m=n$ we get $$0\le \left(\alpha_j^k \cdot  (\mu_j^*H)^l \cdot (\mu_j^*(H-\a))^m\right)\leq (H^n).$$ By expanding the right hand side of (\ref{volestapp}) we thus get 
\begin{eqnarray} \label{soonfin}
\vol(\a)-(\a_j^n)\ge nt(\a_j^{n-1} \cdot \mu_j^*H)-nt(\a_j^{n-1}\cdot (\mu_j^*(H-\a)+\a_j))-Ct^2= \nonumber\\=nt(\a_j^{n-1}\cdot (\mu_j^*\a-\a_j))-Ct^2=nt(\a_j^{n-1}\cdot E_j)-Ct^2,
\end{eqnarray} 
where $C>0$ is some uniform constant.

Note that since $\a_j$ is K\"ahler $(\a_j^{n-1}\cdot E_j)\geq 0$. By assumption $\lim_{j\to\infty}(\a_j^{n-1}\cdot E_j)=(\langle \a^{n-1}\rangle \cdot \a)-\vol(\a)$, thus $(\a_j^{n-1}\cdot E_j)$ is at least bounded by some uniform constant. Without loss of generality we can assume that $(\a_k^{n-1}\cdot E_k)\leq \frac{2C}{n}$.

Setting $t=n\frac{(\a_j^{n-1}\cdot E_j)}{2C}\in [0,1]$ in (\ref{soonfin}) gives the estimate
$$
(\a_j^{n-1}\cdot E_j)^2\le \frac{4C}{n^2}\left(\vol(\a)-(\a_j^n)\right).
$$
By assumption $\lim_{j\to \infty} (\a_j^n)=\vol(\a)$, proving as desired that $\lim_{j\to \infty}(\a_j^{n-1}\cdot E_j)= 0$. 

\end{proof}

\bigskip
\noindent David Witt Nystr\"om \newline
Department of Mathematical Sciences \newline
Chalmers University of Technology and the University of Gothenburg \newline
SE-412 96 Gothenburg, Sweden \newline
wittnyst@chalmers.se, danspolitik@gmail.com

\bigskip
\noindent S\'ebastien Boucksom \newline
CNRS--CMLS\newline
 \'Ecole Polytechnique\newline
  F-91128 Palaiseau Cedex, France\newline
sebastien.boucksom@polytechnique.edu

\end{document}